\documentclass{amsart}
\usepackage[utf8]{inputenc}

\title[Computable reducibility and an effective jump operator]{Computable reducibility of equivalence relations and an effective jump operator}
\author{John D. Clemens}
\address{1910 University Drive, Boise, ID 83725}
\email{johnclemens@boisestate.edu}
\author{Samuel Coskey}
\address{1910 University Drive, Boise, ID 83725}
\email{scoskey@boisestate.edu}
\author{Gianni Krakoff}
\address{1910 University Drive, Boise, ID 83725}
\email{giannikrakoff@boisestate.edu}

\usepackage{bm,amssymb,cancel}

\renewcommand{\setminus}{-}
\newcommand{\N}{{\mathbb N}}

\newcommand{\id}{\mathsf{Id}}

\DeclareMathOperator{\ran}{\mathop{\mathrm{ran}}}
\DeclareMathOperator{\supp}{\mathop{\mathrm{supp}}}

\newcommand{\jump}{{\dot{+}}}

\newtheorem{theorem}{Theorem}[section]
\newtheorem{lemma}[theorem]{Lemma}
\newtheorem{proposition}[theorem]{Proposition}
\newtheorem{corollary}[theorem]{Corollary}
\newtheorem{claim}{Claim}
\theoremstyle{definition}
\newtheorem{definition}[theorem]{Definition}
\theoremstyle{remark}

\newtheorem{question}{Question}

\newenvironment{claimproof}[1][\proofname]{\par
  \normalfont
\trivlist 
  \item[\hskip\labelsep\hskip\parindent
    #1:]\ignorespaces
}{ \ifmmode 
  \else \leavevmode\unskip\penalty9999 \hbox{}\nobreak\hfill
  \fi
          \quad\hbox{\text{$\Box$(claim)}} \endtrivlist }

\begin{document}

\begin{abstract}
  We introduce the computable FS-jump, an analog of the classical Friedman--Stanley jump in the context of equivalence relations on the natural numbers. We prove that the computable FS-jump is proper with respect to computable reducibility. We then study the effect of the computable FS-jump on computably enumerable equivalence relations (ceers).
\end{abstract}

\maketitle

\section{Introduction}

The backdrop for our study is the notion of computable reducibility of equivalence relations. If $E,F$ are equivalence relations on $\N$ we say $E$ is \emph{computably reducible} to $F$, written $E\leq F$, if there exists a computable function $f\colon\N\to\N$ such that for all $n$,$n'$
\[n\mathrel{E}n'\iff f(n)\mathrel{F}f(n')\text{.}
\]
This notion was first studied in both \cite{ershov,bernardi-sorbi}; it has recently garnered further study for instance in \cite{gao-gerdes,fokina-friedman-classes,fokina-friedman-etal,coskey-hamkins-miller} and numerous other works including those cited below.

Computable reducibility of equivalence relations may be thought of as a computable analog to Borel reducibility of equivalence relations on standard Borel spaces. Here if $E,F$ are equivalence relations on standard Borel spaces $X,Y$ we say $E$ is \emph{Borel reducible} to $F$, written $E\leq_B F$, if there exists a Borel function $f\colon X\to Y$ such that $x\mathrel{E}x'\iff f(x)\mathrel{F}f(x')$. We refer the reader to \cite{gao} for the basic theory of Borel reducibility.

One of the major goals in the study of computable reducibility is to compare the relative complexity of classification problems on a countable domain. In this context, if $E\leq F$ we say that the classification up to $E$-equivalence is no harder than the classification up to $F$-equivalence. For instance, classically the rank~$1$ torsion-free abelian groups (the subgroups of $\mathbb{Q}$) may be classified up to isomorphism by infinite binary sequences up to almost equality. Since this classification may be carried out in a way which is computable in the indices, there is a computable reduction from the isomorphism equivalence relation on c.e.\ subgroups of $\mathbb{Q}$ to the almost equality equivalence relation on c.e.\ binary sequences.

A second major goal in this area is to study properties of the hierarchy of equivalence relations with respect to computable reducibility. The computable reducibility quasi-order is quite complex: for instance it is shown in \cite[Theorem~4.5]{bard} that it is at least as complex as the Turing degree order, and in \cite{andrews2021structure} that its theory is equivalent to second order arithmetic. In a portion of this article we will pay special attention to the sub-hierarchy consisting of just the ceers. An equivalence relation $E$ on $\N$ is called a \emph{ceer} if it is computably enumerable, as a set of pairs. Ceers were called positive equivalence relations in \cite{ershov}, subsequently named ceers in \cite{gao-gerdes}, and further studied in works such as \cite{andrews-etal-universal,andrews-sorbi-joins,andrews-sorbi-jumps}.

As with other complexity hierarchies, it is natural to study operations such as jumps. One of the most important jumps in Borel complexity theory is the Friedman--Stanley jump, which is defined as follows. If $E$ is a Borel equivalence relation on the standard Borel space $X$, then the \emph{Friedman--Stanley jump} of $E$, denoted $E^{+}$, is the equivalence relation defined on $X^\N$ by
\[x\mathrel{E^{+}}x'\iff\{[x(n)]_E:n\in\N\}=\{[y(n)]_E:n\in\N\}.
\]
Friedman and Stanley showed in \cite{friedman-stanley} that the jump is \emph{proper}, that is, if $E$ is a Borel equivalence relation, then $E<_B E^{+}$. Moreover they studied the hierarchy of iterates of the jump and showed that any Borel equivalence relation induced by an action of $S_{\infty}$ is Borel reducible to some iterated jump of the identity.

In this article we study a computable analog of the Friedman--Stanley jump, called the computable FS-jump and denoted $E^{\jump}$, in which the arbitrary sequences $x(n)$ are replaced by computable enumerations $\phi_e(n)$. In Section~2 we will give the formal definition of the computable FS-jump, and establish some of its basic properties.

In Section~3 we show that the computable FS-jump is proper, that is, if $E$ is a hyperarithmetic equivalence relation, then $E<E^{\jump}$. We do this by showing that any hyperarithmetic set is many-one reducible to some iterated jump of the identity, and establishing rough bounds on the descriptive complexity of these iterated jumps.

In Section~4 we study the effect of the computable FS-jump on ceers. We show that if $E$ is a ceer with infinitely many classes, then $E^{\jump}$ is bounded below by the identity relation $\id$ on $\N$, and above by the equality relation $=^{ce}$ on c.e.\ sets. This leads to a natural investigation of the structure that the jump induces on the ceers, analogous to the study of the structure that the Turing jump induces on the c.e.\ degrees. For instance, we may say that a ceer $E$ is \emph{high for the computable FS-jump} if $E^{\jump}$ is computably bireducible with $=^{ce}$. At the close of the section, we begin to investigate the question of which ceers are high for the computable FS-jump and which are not.

In the final section we present several open questions arising from these results.

\textbf{Acknowledgement.} This work includes a portion of the third author's master's thesis \cite{gianni-thesis}. The thesis was written at Boise State University under the supervision of the first and second authors. The authors would also like to thank the referee for suggesting numerous improvements.

\section{Basic properties of reducibility and the jump}

In this section we fix some notation, introduce the computable FS-jump, and exposit some of its basic properties.

In this and future sections, we will typically use the letter $e$ for an element of $\N$ which we think of as an index for a Turing program. We will use $\phi_e$ for the partial computable function of index $e$, and $W_e$ for the domain of $\phi_e$.

\begin{definition}
  \label{def:jump}
  Let $E$ be an equivalence relation on $\N$. The \emph{computable FS-jump} of $E$ is the equivalence relation on indices of c.e.\ subsets of $\N$ defined by
  \[e\mathrel{E^{\jump}}e'\iff
    \{[\phi_e(n)]_E:n\in\N\}=\{[\phi_{e'}(n)]_E:n\in\N\}.
  \]
  When $E$ is defined on a countable set other than $\N$ (or computable subset thereof) we define $E^{\jump}$ similarly, considering $\varphi_e$ to have its range in the domain of $E$; formally we may compose $\varphi_e$ with a computable bijection from $\N$ to the domain of $E$.
  Furthermore we define the iterated jumps $E^{\jump n}$ inductively by $E^{\jump 1}=E^{\jump}$ and $E^{\jump (n+1)}=(E^{ \jump n})^{\jump}$.
\end{definition}

We remark that we could also have defined $E^{\jump}$ by working with domains $W_e$ rather than ranges $\ran(\phi_e)$. While each choice has conveniences, we use Definition~\ref{def:jump} due to its analogy with the Friedman--Stanley jump.

We mention here that several other jumps of equivalence relations have been studied in the case of ceers. The halting jump and saturation jump were introduced in \cite{gao-gerdes}. The halting jump of $E$, denoted $E'$, is defined by setting $x \mathrel{E'} y$ iff $x = y \vee \phi_x(x) \downarrow \mathrel{E} \varphi_y(y) \downarrow$. The halting jump and its transfinite iterates are investigated extensively in \cite{andrews-sorbi-jumps}. The saturation jump of $E$, denoted $E^{+}$, is defined on finite subsets of $\N$ where $x$ and $y$ are saturation jump equivalent if their $E$-saturations are equal as sets. The saturation jump may be viewed as a finite-sequence version of the computable FS-jump. As observed in \cite{gao-gerdes} it is not always the case that $E < E'$ and $E < E^{+}$. It is worth noting that the computable FS-jump dominates the saturation jump under computable reducibility, and dominates the halting jump for ceers $E$.

Unless explicitly stated otherwise, any further use of the word ``jump'' will refer to the computable FS-jump.

We are now ready to establish some of the basic properties of the computable FS-jump. In the following, we let $\id$ denote the identity equivalence relation on $\N$. It is worth noting that, although several of these results are direct analogues of results in Section~7 of \cite{gao-gerdes}, our results apply to an arbitrary equivalence relation $E$ and not only ceers (unless stated otherwise).

\begin{proposition}
  \label{prop:monotone}
  For any equivalence relations $E$ and $F$ on $\N$ we have:
  \begin{enumerate}
    \item $E\leq E^{\jump}$.
    \item If $E$ has only finitely many classes, then $E < E^{\jump}$.
    \item If $E\leq F$ then $E^{\jump}\leq F^{\jump}$.
  \end{enumerate}
\end{proposition}

\begin{proof}
  (a) Let $f$ be a computable function such that for all $e$ we have that $\phi_{f(e)}$ is the constant function with value $e$. (To see that there is such a computable function $f$, one can either ``write a Turing program'' for the machine indexed by $f(e)$ or employ the s-m-n theorem. In the future we will not comment on the computability of functions of this nature.) Then $e\mathrel{E}e'$ if and only if $[e]_E=[e']_E$, if and only if $f(e)\mathrel{E^{\jump}}f(e')$.
  
  (b) Note that if $E$ has $n$ classes, then $E^{\jump}$ has $2^n$ classes.

  (c) This is similar to \cite[Theorem~8.4]{gao-gerdes}. Let $f$ be a computable reduction from $E$ to $F$. Let $g$ be a computable function such that $\phi_{g(e)}(n)=f(\phi_e(n))$. Then it is straightforward to verify that $g$ is a computable reduction from $E^{\jump}$ to $F^{\jump}$.
\end{proof}

Slightly less trivially we also note the following.

\begin{proposition}
  \label{prop:double_plus}
  For any $E$ with infinitely many classes we have $\id\leq E^{\jump\jump}$.
\end{proposition}

\begin{proof}
  We define a reduction function $f$ that works simultaneously for all equivalence relations $E$ with infinitely many classes. Given $n$, let $f(n)$ be a code for a machine such that the sequence of sets $S_i=\phi_{f(n)}(i)$ consists of all $n$-element subsets of $\N$. Clearly since $E^{\jump\jump}$ is reflexive we have that $n=n'$ implies $f(n)\mathrel{E^{\jump\jump}}f(n')$. Conversely suppose $n\neq n'$, and assume without loss of generality that $n<n'$. Then for all $i\in\N$ we have that $[\phi_{f(n)}(i)]_{E^{\jump}}$ is a code for at most $n$-many $E$-classes. On the other hand since $E$ has infinitely many classes, there exists $i\in\N$ such that $[\phi_{f(n)}(i)]_{E^{\jump}}$ is a code for exactly $n'$-many $E$-classes. It follows that $\{[\phi_{f(n)}(i)]_{E^{\jump}}:i\in\N\}\neq\{[\phi_{f(n')}(i)]_{E^{\jump}}:i\in\N\}$, or in other words, $f(n)\mathrel{\cancel{E^{\jump\jump}}}f(n')$.
\end{proof}

In the following, we let $E\oplus F$ denote the equivalence relation defined on $\N\times\{0,1\}$ by $(m,i)(E\oplus F)(n,j)$ iff $(i=j=0)\wedge(m\mathrel{E}n)$ or $(i=j=1)\wedge(m\mathrel{F}n)$. Finally, we let $E\times F$ denote the equivalence relation defined on $\N\times\N$ by $(m,n)(E\times F)(m',n')$ iff $m\mathrel{E}m'\wedge n\mathrel{F}n'$.

\begin{proposition}
  \label{prop:product}
  $(E\oplus F)^{\jump}$ is computably bireducible with $E^{\jump}\times F^{\jump}$.
\end{proposition}
  
\begin{proof}
  For the forward reduction, given an index $e$ for a function into $\N\times\{0,1\}$, let $\phi_{e_0}(n)=m$ if $\phi_{e}(n)=(m,0)$ and let $\phi_{e_1}(n)=m$ if $\phi_{e}(n)=(m,1)$; $\phi_{e_i}$ is undefined otherwise. Then the map $e\mapsto(e_0,e_1)$ is a reduction from $(E\oplus F)^{\jump}$ to $E^{\jump}\times F^{\jump}$. For the reverse reduction, given a pair of indices $(e_0,e_1)$ we define $\phi_e(2n)=(\phi_{e_0}(n),0)$ and $\phi_e(2n+1)=(\phi_{e_1}(n),1)$. Once again it is easy to verify $(e_0,e_1)\mapsto e$ is a reduction from $E^{\jump}\times F^{\jump}$ to $(E\oplus F)^{\jump}$.
\end{proof}

In the next result we will briefly consider the connection between the computable FS-jump and the restriction of the classical FS-jump to c.e.\ sets. In the literature, the $n$th iterated classical FS-jump of $\id$ is usually denoted $F_n$. For our purposes it will be convenient to regard each $F_n$ as an equivalence relation on $\mathcal P(\N)$. Thus we officially define $F_1$ as the equality relation on $\mathcal P(\N)$. Letting $\langle\cdot,\cdot\rangle$ be the usual pairing function $\N^2\to\N$, and let $A^{[n]}$ denote the $n$th ``column'' of $A$, that is, $A^{[n]}=\{p\in\N:\langle n,p\rangle\in A\}$. We then officially define $A\mathrel{F_2}B$ iff $\{A^{[n]}:n\in\N\}=\{B^{[n]}:n\in\N\}$. Similarly for all $n$ we can officially define $F_n$ on $\mathcal P(\N)$ by means of a fixed uniformly computable family of bijections between $\N^n$ and $\N$. So defined, $F_n$ is naturally Borel bireducible with the literal $n$th iterated classical FS-jump of $\id$.

Next, recall from \cite{coskey-hamkins-miller} that for any equivalence relation $E$ on $\mathcal P(\N)$ we can define its \emph{restriction to c.e.\ sets} $E^{ce}$ on $\N$ by
\[e\mathrel{E}^{ce}e'\iff W_e\mathrel{E}W_{e'}.
\]
In particular, $(F_1)^{ce}$ is $=^{ce}$, which figures prominently in the theory of computable reducibility. We are now ready to state the following.

\begin{proposition}
  \label{prop:idplus}
  For any $n$, we have that $\id^{\jump n}$ is computably bireducible with $(F_n)^{ce}$.
\end{proposition}

\begin{proof}[Proof sketch]
  For $n=1$, we need to show that $\id^{\jump}$ is computably bireducible with $=^{ce}$, which amounts to the effective equivalence of a c.e. set being either the domain or the range of a partial computable function. Namely, let $f$ and $g$ be computable functions so that $W_{f(e)}=\ran(\varphi_e)$ and $\ran(\varphi_{g(e)})=W_e$; then $f$ and $g$ provide the respective reductions.   
  For the induction step, it is sufficient to show that for any $n$ we have that $((F_n)^{ce})^{\jump}$ is computably bireducible with $(F_{n+1})^{ce}$. For notational simplicity, we briefly illustrate this just in the case when $n=1$. For the reduction from $((F_1)^{ce})^{\jump}$ to $(F_2)^{ce}$, we define $f$ to be a computable function such that for all $n$ we have $(W_{f(e)})^{[n]}=W_{\phi_e(n)}$. For the reduction from $(F_2)^{ce}$ to $((F_1)^{ce})^{\jump}$, we define $g$ to be a computable function such that for all $n$ we have $(W_{\phi_{g(e)}})^{[n]}=(W_e)^{[n]}$.
\end{proof}

We shall make frequent use of the particular case that $\id^{\jump}$ is computably bireducible with $=^{ce}$.

To conclude the section, we define transfinite iterates of the computable FS-jump. The transfinite jumps allow one to extend results such as the previous proposition into the transfinite, and they also play a key role in the next section. For the definition, recall that Kleene's $\mathcal O$ consists of notations for ordinals and is defined as follows: $1\in\mathcal O$ is a notation for $0$, if $a\in\mathcal O$ is a notation for $\alpha$ then $2^a$ is a notation for $\alpha+1$, and if for all $n$ we have $\phi_e(n)$ is a notation for $\alpha_n$ with the notations increasing in $\mathcal{O}$ with respect to $n$, then $3\cdot 5^e$ is a notation for $\sup_n\alpha_n$. We refer the reader to \cite{sacks} for background on $\mathcal O$.

\begin{definition}
  We define $E^{\jump a}$ for $a\in\mathcal O$ recursively as follows.
  \begin{align*}
    E^{\jump 1}&=E\\
    E^{\jump 2^b}&=(E^{\jump b})^{\jump}\\
    E^{\jump 3\cdot 5^e}&=\{(\langle m,x\rangle,\langle n,y\rangle):(m=n)\wedge (x\mathrel{E}^{\jump \phi_e(m)}y)\}
  \end{align*}
\end{definition}

We remark that it is straightforward to extend Proposition~\ref{prop:idplus} into the transfinite as follows. Given a notation $a\in\mathcal O$ for $\alpha$, we may use $a$ to define an equivalence relation $F_a$ on $\mathcal P(\N)$ which is Borel bireducible with the $\alpha$-iterated FS-jump $F_\alpha$. We then have that $\id^{\jump a}$ is computably bireducible with $(F_a)^{ce}$. We do not know, however, whether $\id^{\jump a}$ and $\id^{\jump a'}$ are computably bireducible when $a$ and $a'$ are different notations for the same ordinal.

The following propositions will be used in the next section.

\begin{proposition}
  \label{prop:closed}
  If $E^{\jump} \leq E$ then for any $a \in \mathcal{O}$ we have $E^{\jump a} \leq E$.
\end{proposition}
   
\begin{proof}
  We proceed by recursion on $a \in \mathcal{O}$. It follows from our hypothesis together with Proposition~\ref{prop:monotone}(b) that $E$ has infinitely many classes. By Proposition~\ref{prop:double_plus}, we have $\id \leq E^{\jump\jump}$ and hence $\id \leq E$. From this we can see that $E\times\id\leq E^{\jump\jump}$ as follows. Suppose $h\colon\id \leq E$ and define $h'$ by arranging for $W_{h'(e,n)}=\{0,1\}$, $\phi_{h'(e,n)}(0)=$ a code for $\{e\}$, and $\phi_{h'(e,n)}(1)=$ a code for $\{h(n),h(n+1)\}$. Since $h(n)$ and $h(n+1)$ are distinct for each $n$, we can distinguish $\{h(n),h(n+1)\}$ from $\{e\}$ and recover $e$ and $n$ from $h'(e,n)$, so that $h'\colon E \times \id \leq E^{\jump\jump}$.  Hence we have $E \times \id \leq E$, and we may fix a computable reduction function $g\colon E \times \id \leq E$.
  
  Now let $f\colon E^{\jump} \leq E$ and define uniformly $f_a : E^{\jump a} \leq E$ as follows. Let $f_1$ be the identity map. Given $f_a:E^{\jump a}\leq E $ apply Proposition \ref{prop:monotone}(c) to get $f^+_a:(E^{\jump a})^{\jump}\leq E^{\jump}$, then define $f_{2^a}=f\circ f^+_a$. To define $f_{3 \cdot 5^e}$ it suffices to find a reduction from $E^{\jump 3 \cdot 5^e}$ to $E \times \id$ and compose with $g$; this follows from the fact that we have each $E^{\jump \varphi_e(n)}$ uniformly reducible to $E$ by the effectiveness of the recursion.
\end{proof}

\begin{proposition}
If $E \times \id \leq E$ then for any $a \in \mathcal{O}$ we have $E^{\jump a} \times \id \leq E^{\jump a}$.
\end{proposition}

\begin{proof}
  We proceed by recursion on $a \in \mathcal{O}$, noting that the induction will produce the reduction functions effectively from $a$. Suppose first that $E^{\jump a} \times \id \leq E^{\jump a}$. Then $E^{\jump 2^a} \times \id = (E^{\jump a})^{\jump} \times \id \leq (E^{\jump a})^{\jump} \times \id^{\jump}$, which is bireducible with $(E^{\jump a} \oplus \id)^{\jump}$ by Proposition~\ref{prop:product}. Since the hypothesis implies $\id \leq E$, this is reducible to $(E^{\jump a} \oplus E^{\jump a})^{\jump}$, which is reducible to $(E^{\jump a} \times \id)^{\jump}$, and hence reducible to $(E^{\jump a})^{\jump}=E^{\jump 2^a}$.
  For $E^{\jump 3 \cdot 5^e}$, we assume that $E^{\jump \varphi_e(m)} \times \id \leq E^{\jump \varphi_e(m)}$ uniformly in $m$, from which we see that $E^{\jump 3 \cdot 5^e} \times \id \leq (E \times \id)^{\jump 3 \cdot 5^e} \leq E^{\jump 3 \cdot 5^3}$.
\end{proof}

Since a computable bijection from $\N \times \N$ to $\N$ shows $\id \times \id \leq \id$, we get:

\begin{corollary}
  \label{cor:absorbs_id}
    For any $a \in \mathcal{O}$ we have $\id^{\jump a} \times \id \leq \id^{\jump a}$.
\end{corollary}

\section{Properness of the jump}

In this section we establish the following main result.

\begin{theorem}
  \label{thm:proper}
  If $E$ is a hyperarithmetic equivalence relation on $\N$, then $E<E^{\jump}$.
\end{theorem}

Since we have $E \leq E^{\jump}$ for each equivalence relation $E$, this amounts to showing that no hyperarithmetic equivalence relation is a fixed point of the computable FS-jump. We will in fact establish the following stronger result.

\begin{theorem}
\label{thm:fixed_points}
Let $E$ be an equivalence relation on $\N$ which is a fixed point for the computable FS-jump. Then $E$ is an upper bound in the $m$-degrees for all hyperathmetic sets.
\end{theorem}

The proof will proceed by showing that iterated jumps of the identity have cofinal descriptive complexity among hyperarithmetic sets. Specifically, we will show that every hyperarithmetic set is many-one reducible to $\id^{\jump a}$ for some $a \in \mathcal{O}$. The proof will involve an induction on the hyperarithmetic hierarchy, and we will utilize a particular type of many-one reduction which we now introduce.

\begin{definition}
Given a relation $E$ on $\N$, we define the relation $\subseteq_E$ by setting $e \subseteq_E e'$ if the following holds:
  \[\forall n [ \phi_e(n) \downarrow \ \Rightarrow \exists m ( \phi_{e'}(m) \downarrow \wedge \phi_e(n) \mathrel{E} \phi_{e'}(m))].
  \]
\end{definition}
We write $e \supseteq_E e'$ when $e' \subseteq_E e$.
Note that when $E$ is an equivalence relation, $\subseteq_E$ is a quasi-order and we have $e\mathrel{E^{\jump}}e'$ iff $e\subseteq_E e'$ and $e'\subseteq_E e$. 

\begin{definition}
Given a set $P$ and an equivalence relation $E$, we say that \emph{$P$ is subset-reducible to $E^{\jump}$} if there is a computable function $h$ and $e_0 \in \N$ so that
for all $n$ we have $h(n) \subseteq_E e_0$, and $P(n) \iff h(n) \mathrel{E^{\jump}} e_0$. We call the pair $(h, e_0)$ a \emph{subset-reduction}.
\end{definition}

If $P$ is subset-reducible to $E^{\jump}$ then it is clearly many-one reducible; we will show that every hyperarithmetic set is subset-reducible to some iterated jump of $\id$.
Since in general $P$ may be many-one reducible to $E$ without $P^c$ being reducible to an iterated jump of $E$, we wish to only use ``positive'' induction steps, i.e., an inductive construction of the hyperarithmetic sets starting from computable sets and involving only effective unions and intersections. Also, since we need to uniformly produce reducing functions throughout the construction of a set, we want to consider the entire construction at once. To this end we introduce the notion of a computable Borel code for a hyperarithmetic set. There are many different presentations of computable Borel codes, all of which give the same collection of sets; the following definition is a slight variation of that given in Chapter~27 of \cite{miller}.

\begin{definition}
  A \emph{computable Borel code} is a pair $(T,f)$ where $T$ is a computable well-founded tree on $\N$ so that $t \smallfrown n \in T$ for all $n$ for non-terminal nodes $t$, and $f$ is a computable function from the terminal nodes of $T$ to $\N$. Given a computable Borel code $(T,f)$, the set $B(T,f)$ is defined by recursion on $t \in T$ as follows. If $t$ is a terminal node, then $B_t(T,f)=\ran \phi_{f(t)}$, and if $t$ is not a terminal node, then  $B_t(T,f)=\{ n: \forall p \exists q (n \in B_{t \smallfrown \langle p,q \rangle}(T,f) )\}$. We let $B(T,f)=B_{\emptyset}(T,f)$.
\end{definition}

The following characterization then follows from the fact that a set is hyperarithmetic if and only if it is $\Delta^1_1,$ together with the Kleene Separation Theorem and the hyperarithmetic codes used in its proof (see, e.g., \cite[Chapter II]{sacks} and \cite[Theorem 27.1]{miller}).

\begin{theorem}
  A set $B$ is hyperarithmetic if and only if there is a computable Borel code $(T,f)$ such that $B=B(T,f)$.
\end{theorem}

From this characterization, we see that it will suffice to consider three types of inductive steps as described in Lemma~\ref{lem:proper_union}, Lemma~\ref{lem:proper_intersection}, and Lemma~\ref{lem:proper_limit}. We begin by considering the case of a $\Sigma^0_3$ set because it allows us to produce slightly better complexity bounds, as discussed later in this section, and introduces key ideas used in the subsequent proofs. 

In the following, we will say that $e$ is \emph{an index for an enumeration} of the c.e.\ set $W$ if $\ran \phi_e = W$.    We will repeatedly utilize the fact that we can effectively enumerate the $E^{\jump}$-classes of the c.e. supersets of a given set, i.e., $\{[e]_{E^{\jump}} : e \supseteq_{E} e_0\} = \{[W_i \cup e_0]_{E^{\jump}} : i \in \N\}$, where we use $W_i \cup e_0$ to denote an index for an enumeration of $\ran \phi_{e_0} \cup W_i$. The analogous statement with $\subseteq_E$ replacing $\supseteq_E$ does not hold, as illustrated in Proposition~\ref{prop:counterexample} below, which is why we repeat this process twice to handle existential quantification.  Recall that $=^{ce}$ is computably bireducible with $\id^{\jump}$.

\begin{lemma}
\label{lem:proper_base}
Let $P$ be $\Sigma^0_3$. Then $P$ is subset-reducible to $\left(=^{ce}\right)^{\jump}$.
\end{lemma}

\begin{proof}
Choose $i_0$ with
 $P(n) \iff \exists q \forall m\ \phi_{i_0}(\langle q,m,n\rangle) \downarrow$,
 so that
 \[ P(n) \iff  \exists q\  \{ \langle q,m,n \rangle : m \in \N \} \subset W_{i_0} .\]
 Letting $W_{g(q,n)} = W_{i_0} \cup \{ \langle q,m,n\rangle : m \in \N\}$, we then have
 \[ P(n) \iff \exists q\  W_{g(q,n)} = W_{i_0} ,\]
 with  $W_{g(q,n)} \supset W_{i_0}$ for all $q$ and $n$. Then
 \[ P(n)  \iff \exists q\  \{W_i \cup W_{g(q,n)} : i \in \N\} = \{W_i \cup W_{i_0} : i \in \N \}, \]
 with $\{W_i \cup W_{g(q,n)} : i \in \N\} \subset \{W_i \cup W_{i_0} : i \in \N \}$ for all $q$ and $n$. Hence
 \[ P(n) \iff \{ W_i \cup W_{g(q,n)} : i \in \N \wedge q \in \N\} = \{W_i \cup W_{i_0} : i \in \N \}, \]
 with $\{ W_i \cup W_{g(q,n)} : i \in \N \wedge q \in \N\} \subset \{W_i \cup W_{i_0} : i \in \N \}$ for all $q$ and $n$, and equality holding only when there is $q$ with  $W_{g(q,n)} = W_{i_0}$. 
 
 Let $h(n)$ be such that $\phi_{h(n)}(\langle i,q\rangle)$ is an index for an enumeration of $W_i \cup W_{g(q,n)}$ and let $e_0$ be such that $\phi_{e_0}(i)$ is an index for an enumeration of $W_i \cup W_{i_0}$. Then we have
 $P(n) \iff h(n)  \mathrel{\left(=^{ce}\right)^{\jump}} e_0$,
with $h(n) \subseteq_{=^{ce}} e_0$ for all $n$.
 \end{proof}

\begin{lemma}
\label{lem:proper_union}
Suppose $Q$ is subset-reducible to $E^{\jump}$, and $P(n) \iff \exists q\  Q(\langle q,n\rangle)$. Then $P$ is subset-reducible to $E^{\jump\jump\jump}$. Moreover, there are computable functions $\Psi$ and $\chi$ so that if $(\phi_i,d_0)$ is a subset-reduction from $Q$ to $E^{\jump}$, then $(\phi_{\Psi(i)}, \chi(d_0))$ is a subset-reduction from $P$ to $E^{\jump \jump \jump}$.
\end{lemma}

\begin{proof}
   Let $(f,d_0)$ be a subset-reduction from $Q$ to $E^{\jump}$. We then have:   
   \begin{align*}
      P(n) &\iff \exists q\  Q(\langle q,n\rangle) \\
      & \iff \exists q\ f(\langle q,n \rangle) \mathrel{E^{\jump}} d_0 \\
      & \iff \exists q\ \{ [m]_E : m \in \ran \phi_{f(\langle q,n \rangle)} \} = \{ [m]_E : m \in \ran \phi_{d_0}\} \\
      & \iff \exists q\ \{[e]_{E^{\jump}} : e \supseteq_E f(\langle q,n\rangle) \} = \{[e]_{E^{\jump}} : e \supseteq_E d_0\},
    \end{align*}
   with $\{[e]_{E^{\jump}} : e \supseteq_E f(\langle q,n\rangle) \} \supset \{[e]_{E^{\jump}} : e \supseteq_E d_0\}$ for all $q$ and $n$. Let $j$ be such that $\phi_{j(n,q)}(i)$ is an index for an enumeration of $W_i \cup \ran \phi_{f(\langle q,n\rangle)}$ for each $n$, $q$, and $i$, and let $j_0$ be such that $\phi_{j_0}(i)$ is an index for an enumeration of $W_i \cup \ran \phi_{d_0}$ for each $i$. Then we have
   \[ P(n) \iff \exists q\ j(n,q) \mathrel{E^{\jump \jump}} j_0 ,\]
   with $j(n,q) \supseteq_{E^{\jump}} j_0$ for all $n$ and $q$. Hence
   \[ P(n) \iff \exists q\ \{[e]_{E^{\jump \jump }} : e \supseteq_{E^{\jump}} j(n,q) \} = \{[e]_{E^{\jump \jump }} : e \supseteq_{E^{\jump}} j_0\},\]
   with $\{[e]_{E^{\jump \jump }} : e \supseteq_{E^{\jump}} j(n,q) \} \subset \{[e]_{E^{\jump \jump }} : e \supseteq_{E^{\jump}} j_0\}$ for all $n$ and $q$. Hence we also have $\{[e]_{E^{\jump \jump }} :\exists q\ e \supseteq_{E^{\jump}} j(n,q) \} \subset \{[e]_{E^{\jump \jump }} : e \supseteq_{E^{\jump}} j_0\}$ for all $n$, and we claim that
   \[ P(n) \iff \{[e]_{E^{\jump \jump }} : \exists q\ e \supseteq_{E^{\jump}} j(n,q) \} = \{[e]_{E^{\jump \jump }} : e \supseteq_{E^{\jump}} j_0\}.\]
   To see this, note if equality holds then $[j_0]_{E^{\jump}}$ must be an element of the left-hand set, so there must be $q_0$ with $j_0 \supseteq_{E^{\jump}} j(n,q_0)$.  Since $j(n,q) \supseteq_{E^{\jump}} j_0$ for all q, we we thus have $j(n,q_0) \mathrel{E^{\jump \jump}} j_0$, so that $P(n)$ holds.
   
   Finally, let $h$ be such that $\phi_{h(n)}(\langle i,q\rangle)$ is an index for an enumeration of $W_i \cup \ran \phi_{j(n,q)}$ for each $n$, $q$, and $i$, an let $e_0$ be such that $\phi_{e_0}(i)$ is an index for an enumeration of $W_i \cup \ran \phi_{j_0}$ for each $i$. Then $h(n) \subseteq_{E^{\jump \jump}} e_0$ for each $n$, and $P(n) \iff h(n) \mathrel{E^{\jump \jump \jump}} e_0$, so that $(h,e_0)$ is a subset-reduction of $P$ to $E^{\jump \jump \jump}$. The construction from $h$ and $e_0$ is uniform in $f$ and $d_0$, so we can produce the functions $\Psi$ and $\chi$ as described.
\end{proof}

\begin{lemma}
\label{lem:proper_intersection}
Suppose $E \times \id \leq E$, $Q$ is subset-reducible to $E^{\jump}$, and $P(n) \iff \forall p  Q(\langle p,n\rangle)$. Then $P$ is subset-reducible to $E^{\jump}$. Moreover, there are computable functions $\Psi$ and $\chi$ so that if $(\phi_i,d_0)$ is a subset-reduction from $Q$ to $E^{\jump}$, then $(\phi_{\Psi(i)}, \chi(d_0))$ is a subset-reduction from $P$ to $E^{\jump}$.
\end{lemma}

\begin{proof}
  Let $(f,d_0)$ be a subset-reduction from $Q$ to $E^{\jump}$, and let $g$ be a reduction from $E \times \id$ to $E$. Define $h$ so that $h(n)$ is an index for an enumeration of 
  $ \{ g(m,p) : m \in \ran \phi_{f(\langle p,n \rangle)} \wedge p \in \N\} $
  and let $e_0$ be an index for an enumeration of
  $ \{ g(m,p) : m \in \ran \phi_{d_0} \wedge p \in \N\}$.
  For all $n$ and $p$ we have $f(\langle p,n\rangle) \subseteq_E d_0$, so that $h(n) \subseteq_E e_0$, and for all $n$ we have:
  \begin{align*}
      P(n) &\iff \forall p\  Q(\langle p,n\rangle) \\
      & \iff \forall p\ f(\langle p,n \rangle) \mathrel{E^{\jump}} d_0 \\
      & \iff \forall p\ \{ [m]_E : m \in \ran \phi_{f(\langle p,n \rangle)} \} = \{ [m]_E : m \in \ran \phi_{d_0}\} \\
      & \iff \{ [(m,p)]_{E \times \id} : m \in \ran \phi_{f(\langle p,n \rangle)} \wedge p \in \N\} = \\
      &\qquad\qquad\qquad \{ [(m,p)]_{E \times \id} : m \in \ran \phi_{d_0} \} \\
      & \iff \{ [g(m,p)]_E : m \in \ran \phi_{f(\langle p,n \rangle)} \wedge p \in \N\} = \\
      &\qquad\qquad\qquad \{ [g(m,p)]_E : m \in \ran \phi_{d_0} \wedge p \in \N \} \\
      & \iff h(n) \mathrel{E^{\jump}} e_0 ,
  \end{align*}
  so that $(h,e_0)$ is a subset-reduction from $P$ to $E^{\jump}$. The construction of $h$ and $e_0$ is uniform, so we can produce the functions $\Psi$ and $\chi$ as described.
\end{proof}

\begin{lemma}
\label{lem:proper_limit}
Suppose $a=3 \cdot 5^e \in \mathcal{O}$, and for each $n$ we have that $(h_n,e_n)$ is a subset-reduction from $A^{[n]}=\{p\in\N:\langle n,p\rangle\in A\}$ to $(E^{\jump \phi_e(n)})^{\jump}$, with the sequences $\langle h_n \rangle_{n \in \N}$ and $\langle e_n \rangle_{n \in \N}$ computable. Then $A$ is subset-reducible to $(E^{\jump a})^{\jump}$. Moreover, there are computable functions $\Psi$ and $\chi$ so that $(\Psi(\langle h_n \rangle_{n \in \N}), \chi(\langle e_n \rangle_{n \in \N}))$ provides the subset-reduction.
\end{lemma}

\begin{proof}
  Define $h$ so that for each $n$ and $p$, $h(\langle n,p\rangle)$ is an index for an enumeration of
  $ \{\langle n, q \rangle : q \in \ran \phi_{h_n(p)} \} \cup \{ \langle m, q \rangle : q \in \ran \phi_{e_m} \wedge m \neq n\}$ ,
  and let $e$ be an index for an enumeration of $\{ \langle m, q \rangle : q \in \ran \phi_{e_m} \wedge m \in \N \}$. Then $(h,e)$ provides the desired subset-reduction since $\{\langle n, q \rangle : q \in \ran \phi_{h_n(p)} \} \subseteq_{E^{\jump a}} \{ \langle n, q \rangle : q \in \ran \phi_{e_n}\}$ for all $n$ and $p$, with $\{\langle n, q \rangle : q \in \ran \phi_{h_n(p)} \} \mathrel{(E^{\jump a})^{\jump}} \{ \langle n, q \rangle : q \in \ran \phi_{e_m}\}$ iff $p\in A^{[n]}$. The existence of $\Psi$ and $\chi$ is clear.
\end{proof}

We now prove the key result for establishing properness of the jump.

\begin{theorem}
\label{thm:key_proper}
For each hyperarithmetic set $B$ there is $a \in \mathcal{O}$ with $B \leq_m \id^{\jump a}$.
\end{theorem}

\begin{proof}
  We will show that for each computable Borel code $(T,f)$ there is $a_T \in \mathcal{O}$ so that $B(T,f) \leq_m \id^{\jump a_T}$.
  For notational convenience we let $B_t=B_t(T,f)$ for $t \in T$. We will recursively define $a_t \in \mathcal{O}$ and let $E_t = \id^{\jump a_t}$ and establish by effective induction on $t \in T$ that $B_t$ is subset-reducible to $E_t^{\jump}$ via $(h_t,e_t)$, with computable maps $t \mapsto a_t$, $t \mapsto h_t$, and $t \mapsto e_t$.
  
For $t$ a terminal node we have $B_t=\ran \phi_{f(t)}$ and we set $a_t=1$ so $E_t=\id$ and $E_t^{\jump}$ is bireducible with $=^{ce}$. Fix a single $e_t$ for all terminal $t$ so that $\ran \phi_{e_t}=\N$, and let $h_t(n)$ be such that $\ran \phi_{h_t(n)} = \N$ if $n \in \ran \phi_{f(t)}$ and $\ran \phi_{h_t(n)} = \emptyset$ if $n \notin \ran \phi_{f(t)}$. Then $(h_t,e_t)$ is a subset-reduction from $B_t$ to $E_t^{\jump}$.

Now let $t$ be a non-terminal node, and assume $a_{t \smallfrown \langle p,q \rangle}$, $h_{t \smallfrown \langle p,q \rangle}$, and $e_{t \smallfrown \langle p,q \rangle}$ have been defined for all $t \smallfrown \langle p,q \rangle \in T$. Fix a computable pairing function $(x,y) \mapsto \langle x,y \rangle$ with computable coordinate functions $(\langle x,y \rangle)_0 =x$ and $(\langle x,y \rangle)_1 =y$, and so that $\langle 0,0 \rangle =0$. Define $R_t$ so that $R_t(\langle q, \langle p,n \rangle \rangle) \iff B_{t \smallfrown \langle p,q \rangle}(n)$, so that $B_t(n) \iff \forall p \exists q\ R_t(\langle q, \langle p,n \rangle \rangle)$.

We first adjust ordinal ranks to produce an increasing sequence so that we can take their supremum in $\mathcal{O}$. Let $\tilde{a}_{t,0}=a_{t \smallfrown \langle 0,0 \rangle}$ and let
\[ \tilde{a}_{t,m+1} = \tilde{a}_{t,m} +_{\mathcal{O}} a_{t \smallfrown \langle (m)_0,(m)_1 \rangle} +_{\mathcal{O}} 2, \]
where $+_{\mathcal{O}}$ is addition in $\mathcal{O}$. Then let $\tilde{a}_{t} = 3 \cdot 5^{i_{t}}$ where $\phi_{i_{t}}(m)= \tilde{a}_{t,m}$ for all $n$.
Observe that if $\psi: E \leq F$ then the map $\tilde{\psi}\colon E^{\jump} \leq F^{\jump}$ as produced in the proof of Proposition~\ref{prop:monotone}(c) will satisfy $e \subseteq_E e' \iff \tilde{\psi}(e) \subseteq_F \tilde{\psi}(e')$. Hence we can uniformly replace $E_{t \smallfrown \langle p,q \rangle}$, $e_{t \smallfrown \langle p,q \rangle}$, and $h_{t \smallfrown \langle p,q\rangle} $ by $\id^{\jump  \tilde{a}_{t,\langle p,q \rangle}}$, a corresponding $\tilde{e}_{t \smallfrown \langle p,q \rangle}$, and a corresponding map $\tilde{h}_{t \smallfrown \langle p,q \rangle}$, respectively, while maintaining the conditions for subset-reductions.

Letting $A_t$ be such that $A_t^{(m)}=B_{t \smallfrown \langle (m)_0,(m)_1\rangle}$ for each $m$, we then can effectively produce a subset-reduction from $A_t$ to $(\id^{\jump \tilde{a}_t})^{\jump}$ by Lemma~\ref{lem:proper_limit}. Since $A_t$ is computably isomorphic to $R_t$ in a uniform way, we can do the same for $R_t$. Letting $S_t(m) \iff \exists q\ R_t(\langle q,m \rangle)$, we then uniformly produce a subset-reduction from $S_t$ to $(\id^{\jump \tilde{a}_t})^{\jump \jump \jump}$ by Lemma~\ref{lem:proper_union}. Recalling that $\id^{\jump a} \times \id \leq \id^{\jump a}$ for all $a \in \mathcal{O}$ by Corollary~\ref{cor:absorbs_id}, we can then apply Lemma~\ref{lem:proper_intersection} to effectively obtain a subset reduction $(h_t,e_t)$ from $B_t$ to $(\id^{\jump \tilde{a}_t})^{\jump \jump \jump}$. Letting $a_t= \tilde{a}_t +_{\mathcal{O}} 2^2$, this completes the induction step for $t$.
\end{proof}

We are now ready to conclude the proof of Theorem~\ref{thm:fixed_points}. Since the hyperarithmetic sets have no hyperarithmetic upper bound in terms of $m$-reducibility, this gives the main theorem of the section, Theorem~\ref{thm:proper}, as an immediate corollary.

\begin{proof}[Proof of Theorem~\ref{thm:fixed_points}]
  Suppose $E^{\jump} \leq E$. By Proposition~\ref{prop:monotone}(b) we can assume that $E$ has infinitely many classes. Thus by Proposition~\ref{prop:double_plus} we have $\id \leq E^{\jump\jump} \leq E$. Hence by Proposition~\ref{prop:closed} we have $\id^{\jump a} \leq E^{\jump a} \leq E$ for all $a \in \mathcal{O}$. But now by Theorem~\ref{thm:key_proper}, every hyperarithmetic set is $m$-reducible to $\id^{\jump a}$ for some $a \in \mathcal{O}$, and hence $m$-reducible to $E$. 
\end{proof}

The proof of Theorem~\ref{thm:key_proper} does not give optimal bounds on the number of iterates of the jump required. With a bit more care, we can show that every $\Pi^0_{\alpha}$ set is reducible to $\id^{\jump a}$ for some $a \in \mathcal{O}$ with $|a|=\alpha$.
 We believe that the optimal bound should be that every $\Pi^0_{2 \cdot \alpha}$ set is reducible to $\id^{\jump a}$ for some $a \in \mathcal{O}$ with $|a|=\alpha$. 
 Lemma~\ref{lem:proper_base} and Lemma~\ref{lem:proper_intersection} show that $\left(=^{ce}\right)^{\jump}$ (and hence $\id^{\jump \jump}$) is $\Pi^0_4$-complete, and we can show by an \emph{ad hoc} argument that $\left(=^{ce}\right)^{\jump\jump}$ is $\Pi^0_6$-complete.  The difficulty is that our induction technique requires two iterates of the jump at each step in order to reverse the direction of set containment twice.
We would prefer to use $\subseteq_E$ rather than $\supseteq_E$ throughout, but we do not see how to effectively enumerate c.e.\ subsets of a given c.e.\ set up to $E^{\jump}$-equivalence, whereas we can enumerate c.e.\ supersets. The natural attempt to do this fails as shown in the following example.

\begin{proposition}
\label{prop:counterexample}
There are $E$ and $e_0$ so that $\{[e]_{E^{\jump}} : e \subseteq_{E} e_0\} \neq \{[W_i \cap e_0]_{E^{\jump}} : i \in \N\}$, where $W_i \cap e_0$ denotes an index for an enumeration of $\ran \phi_{e_0} \cap W_i$.
\end{proposition}

\begin{proof}
Let $E$ be $=^{ce}$, and let $A \subset B$ be c.e.\ sets with $B-A$ not c.e. Let $e_0$ be such that
\[ \ran \phi_{\phi_{e_0}(j)} = \begin{cases} \{k\} & \text{if $j=2k+1$} \\ \{k,k+1\} & \text{if $j=2k+2$} \\ \emptyset & \text{if $j=0$} \end{cases} \]
and let $e$ be such that 
\[ \ran \phi_{\phi_{e}(k)} = \begin{cases} \{k\} & \text{if $k \in B-A$} \\ \{k,k+1\} & \text{if $k \in A$} \\ \emptyset & \text{if $k \notin B$} \end{cases} .\]
Then $e \subseteq_E e_0$ but there is no $i$ with $e \mathrel{E^{\jump}} e_0 \cap W_i$. For if there were, we would have $k \in B-A$ iff $\exists x (x \in W_i \wedge x = \phi_{e_0}(1+2k))$ so that $B-A$ would be c.e.
\end{proof}

We have shown that the computable FS-jump of a hyperarithmetic equivalence relation is always strictly above the relation, so there are no hyperarithmetic fixed points up to bireducibility. If we consider non-hyperarithmetic equivalence relations we can find fixed points of the jump.

\begin{definition}
Let $\cong_{\mathcal T}$ be the isomorphism relation on computable trees.
\end{definition}

Here we can use any reasonable coding of computable trees by natural numbers. Then $\cong_{\mathcal{T}}$ is a $\Sigma_1^1$ equivalence relation which is not hyperarithmetic.
In \cite[Theorem 2]{fokina-friedman-etal} it was shown that $\cong_{\mathcal{T}}$ is $\Sigma_1^1$ complete for computable reducibility, that is, $\cong_{\mathcal{T}}$ is $\Sigma_1^1$ and for every $\Sigma_1^1$ equivalence relation $E$, $E\leq \cong_{\mathcal T}$. We can see that $\cong_{\mathcal T}$ is a jump fixed point, i.e., $\cong_{\mathcal T}^{\jump}$ is computably bireducible with $\cong_{\mathcal T}$. More generally:

\begin{proposition}
Any $\Sigma^1_1$ or $\Pi^1_1$ complete equivalence relation $E$ is a jump fixed point, i.e., $E^{\jump}$ is computably bireducible with $E$.
\end{proposition}

\begin{proof}
 It suffices to show that  $E^{\jump}$ is $\Sigma_1^1$ (resp. $\Pi^1_1$) for any $\Sigma^1_1$ (resp. $\Pi^1_1$ equivalence relation $E$. This follows immediately from the fact that $E^{\jump}$ is a conjunction of $E$ with additional natural number quantifiers. 
\end{proof}

\begin{corollary}
$\cong_{\mathcal T}$ is a jump fixed point.
\end{corollary}

We note that although every hyperarithmetic set is many-one reducible to $\id^{\jump a}$ for some $a \in \mathcal{O}$, we do not know whether every hyperarithmetic equivalence relation $E$ satisfies $E \leq \id^{\jump a}$ for some $a \in \mathcal{O}$.

\section{Ceers and the jump}

Recall from the introduction that $E$ is called a \emph{ceer} if it is a computably enumerable equivalence relation. In this section, we study the relationship between the computable FS-jump and the ceers.

We begin with the following upper bound on the complexity of the computable FS-jump of a ceer. In the statement, recall that if $E$ is an equivalence relation and $W\subset\N$, then $W$ is said to be \emph{$E$-invariant} if it is a union of $E$-equivalence classes.

\begin{proposition}
  \label{prop:upperbound}
  If $E$ is a ceer, then $E^{\jump}\leq\mathord{=}^{ce}$. Moreover, we can find a reduction whose range is contained in the set $\{e\in\N : W_e\text{ is $E$-invariant}\}$.
\end{proposition}

\begin{proof}
  We define a computable function $f$ such that $W_{f(e)}=[\ran\phi_e]_E$. To see that there is such a computable function $f$, one can let $f(e)$ be a program which, on input $n$, searches through all triples $(a,b,c)$ such that $a\in\ran\phi_e$ and $(b,c)\in E$, and halts if and when it finds a triple of the form $(a,a,n)$. Since it is clear that $e\mathrel{E^{\jump}}e'$ if and only if $[\ran\phi_e]_E=[\ran\phi_{e'}]_E$, we have that $f$ is a computable reduction from $E^{\jump}$ to $=^{ce}$.
  It is immediate from the construction that the range of $f$ is contained in $\{e\in\N :  W_e\text{ is $E$-invariant}\}$. 
\end{proof}

The next result gives a lower bound on the complexity of the computable FS-jumps of a ceer.

\begin{theorem}
  \label{thm:lowerbound}
  If $E$ is a ceer with infinitely many equivalence classes, then $\id<E^{\jump}$.
\end{theorem}

\begin{proof}
  We first show that $\id\leq E^{\jump}$. To do so, we first define an auxilliary set of pairs $A$ recursively as follows: Let $(n,j)\in A$ if and only if for every $i<j$ there exists $m<n$ and $(m,i')\in A$ such that $i\mathrel{E}i'$. It is immediate from the definition of $A$, the fact that $E$ is c.e., and the recursion theorem that $A$ is a c.e.\ set of pairs.
  
  We observe that each column $A^{[n]}=\{j:(n,j)\in A\}$ of $A$ is an initial interval of $\N$. It is immediate from the definition that the first column $A^{(0)}$ is the singleton $\{0\}$. Next since $E$ has infinitely many classes, we have that each $A^{[n]}$ is bounded. Moreover $A^{[n]}$ is precisely the interval $[0,j]$ where $j$ is the least value that is $E$-inequivalent to every element of $A^{[m]}$ for all $m<n$.
  
  We now define $f$ to be any computable function such that for all $n$, the range of $\phi_{f(n)}$ is precisely $A^{[n]}$. Then as we have seen, $m<n$ implies there exists an element $j$ in the range of $\phi_{f(n)}$ such that $j$ is $E$-inequivalent to everything in the range of $\phi_{f(m)}$. In particular, $f$ is a computable reduction from $\id$ to $E^{\jump}$.
  
  To establish strictness, assume to the contrary that $E^{\jump}\leq\id$. Then since $E\leq E^{\jump}$, by Theorem~\ref{thm:proper} we have $E<\id$, contradicting that $E$ has infinitely many classes.
\end{proof}

In order to put the previous result in context, we pause our investigation of ceers briefly to consider the question of which $E$ satisfy $\id\leq E^{\jump}$. We first note that it follows from Proposition~\ref{prop:double_plus} that if $E$ is itself a jump, then $\id\leq E^{\jump}$. We now show on the other hand that there exist equivalence relations $E$ such that $\id\not\leq E^{\jump}$. To describe such an equivalence relation, we recall the following notation.

\begin{definition}
  If $A\subset\N$ then the equivalence relation $E_A$ is defined by
  \[m\mathrel{E_A}n\iff m=n\text{ or }m,n\in A\text{.}
  \]
\end{definition}

Thus the equivalence classes of $E_A$ are $A$ itself, together with the singletons $\{i\}$ for $i\notin A$. Note that $E_A\leq E_B$ if and only if $A$ is $1$-reducible to $B$ (see for instance \cite[Proposition~2.8]{coskey-hamkins-miller}).

\begin{theorem}
\label{thm:verydarkarithmetic}
  There exists an arithmetic coinfinite set $A$ such that $\id\not\leq E_A^{\jump}$.
\end{theorem}

\begin{proof}
  Let $P$ be the Mathias forcing poset, that is, $P$ consists of pairs $(s,B)$ where $s\subset\N$ is finite, $B\subset\N$ is infinite, and every element of $s$ is less than every element of $B$. The ordering on $P$ is defined by $(s,B)\leq(t,C)$ if $s\supset t$, $B\subset C$, and $s\setminus t\subset C$.
  
  We first show that if $A^c$ is sufficiently Mathias generic, then $A$ satisfies $\id\not\leq E_A^{\jump}$. In order to do so, let $f$ be any total function so that the sets $\ran\phi_{f(i)}$ are pairwise distinct. Define:
  \begin{align*} 
  D_f = \{(s,B)\in P:\,&(\exists i\neq j)\,[ (s\cup B)\cap(\ran\phi_{f(i)}\triangle\ran\phi_{f(j)})=\emptyset \wedge \\
  &\ran\phi_{f(i)} \cap (s \cup B)^c \neq \emptyset \wedge \ran\phi_{f(i)} \cap (s \cup B)^c \neq \emptyset ]\}.
  \end{align*}
  We claim that $D_f$ is dense in $P$. To see this, let $(s,B)$ be given. Repeatedly applying the pigeonhole principle, we can find infinitely many indices $i_n$ such that the sets $\ran\phi_{f(i_n)}$ agree on $s$. Since the sets $\ran\phi_{f(i)}$ are pairwise distinct, there must be three, $i_0,i_1,i_2$, such that each $\ran\phi_{f(i)}$ is not a subset of $s$. Observe that
  \begin{align*}
  \N = &(\ran\phi_{f(i_0)}\triangle\ran\phi_{f(i_1)})^c
  \cup(\ran\phi_{f(i_0)}\triangle\ran\phi_{f(i_2)})^c \cup \\
  &(\ran\phi_{f(i_1)}\triangle\ran\phi_{f(i_2)})^c.
  \end{align*}
  In particular we can suppose without loss of generality that $i_0$ and $i_1$ satisfy that the set $B'=B\cap(\ran\phi_{f(i_0)}\triangle\ran\phi_{f(i_1)})^c$ is infinite. Then $\ran\phi_{f(i_0)}$ and $\ran\phi_{f(i_1)}$ agree on both $s$ and $B'$. Since neither $\ran\phi_{f(i_0)}$ nor $\ran\phi_{f(i_1)}$ is a subset of $s$, we can remove finitely many elements from $B'$ to ensure that $\ran\phi_{f(i_0)} \cap (s \cup B')^c \neq \emptyset$ and $\ran\phi_{f(i_1)} \cap (s \cup B')^c \neq \emptyset$, and so $(s,B')\in D_f$ completing the claim.
  
  Now let $G\subset P$ be a filter satisfying the following conditions:
  \begin{enumerate}
    \item $G$ meets $\{(s,B)\in P:|s|\geq m\}$ for all $m\in\N$.
    \item $G$ meets $D_f$ for all computable functions $f$ so that the sets $\ran\phi_{f(i)}$ are pairwise distinct.
  \end{enumerate}
  This is possible since the sets in condition~(a) are clearly dense, and we have shown that the $D_f$ are dense. We define the set $A$ by declaring that $A^c=\bigcup\{s:(s,B)\in G\}$. Condition~(a) implies that $A^c$ is infinite, and also that $A^c=\bigcap\{s \cup B : (s,B) \in G\}$; we wish to show that $\id\not\leq E_A^{\jump}$. For this we will show that if $f$ is a given computable function, then $f$ is not a reduction from $\id$ to $E_A^{\jump}$.
  
  Assume, toward a contradiction, that $f$ is a reduction from $\id$ to $E_A^{\jump}$. Then the sets $\ran\phi_{f(i)}$ are pairwise distinct, so there is $(s,B) \in G \cap D_f$. Thus there exist $i\neq j$ such that both $\ran\phi_{f(i)}$ and $\ran\phi_{f(j)}$ intersect $(s \cup B)^c$, and $(s\cup B)\cap(\ran\phi_{f(i)}\triangle\ran\phi_{f(j)})=\emptyset$. Hence both $\ran\phi_{f(i)}$ and $\ran\phi_{f(j)}$ intersect $A$, and $A^c\cap(\ran\phi_{f(i)}\triangle\ran\phi_{f(j)})=\emptyset$. This means that $f(i)\mathrel{E_A^{\jump}}f(j)$, so $f$ is not a reduction from $\id$ to $E_A^{\jump}$, as desired.
  
  Finally, we can ensure $A$ is arithmetic by enumerating the dense sets described above, inductively defining a descending sequence $(s_n,B_n)$ meeting the dense sets, and letting $A^c=\bigcup_n s_n = \bigcap_n (s_n \cup B_n)$. More precisely, note that for any condition $(s,B)$, we can find an extension meeting $D_f$ for a suitable $f$ by intersecting $B$ with a $\Delta^0_2$ set, and the set of $i$ so that $f=\varphi_i$ is suitable is $\Pi^0_3$, so the construction of this sequence may be done computably in $0^{(3)}$, from which we can produce an $A$ which is $\Delta^0_4$ 
\end{proof}

This result leaves open the question of what is the least complexity of an equivalence relation $E$ with infinitely many classes such that $\id\not\leq E^{\jump}$.

Returning to ceers, in view of the bounds from Proposition~\ref{prop:upperbound} and Theorem~\ref{thm:lowerbound}, it is natural to ask whether there is a ceer $E$ such that $E^{\jump}$ lies properly between $\id$ and $=^{ce}$. We first see that there is a large collection of ceers whose jumps are bireducible with $=^{ce}$. We recall the following terminology from \cite{andrews-sorbi-joins}:

\begin{definition}
A ceer $E$ is said to be \emph{light} if $\id\leq E$. $E$ is said to be \emph{dark} if $E$ has infinitely many classes but $\id\not\leq E$.
\end{definition}

Thus every ceer satisfies exactly one of finite, light, or dark.

\begin{proposition}
  \label{prop:light-is-high}
  If $E$ is a light ceer then $E^{\jump}$ is computably bireducible with $=^{ce}$.
\end{proposition}

\begin{proof}
  This is an immediate consequence of Propositions~\ref{prop:monotone}(c), \ref{prop:idplus}, and~\ref{prop:upperbound}.
\end{proof}

We will see that there are also dark ceers which satisfy this conclusion. We introduce the following terminology.

\begin{definition}
  We say a ceer $E$ is \emph{high for the computable FS-jump} if $E^{\jump}$ is computably bireducible with $=^{ce}$.
\end{definition}

This generalizes the notion of lightness for ceers, but also implies that the computable FS-jump is as complicated as possible. As there is no least ceer with infinitely many classes, there does not seem to be a natural notion of low for the computable FS-jump.

In order to describe a dark ceer which is high for the computable FS-jump, recall that a c.e.\ set $A\subset\N$ is called \emph{simple} if there is no infinite c.e.\ set contained in $A^c$. Furthermore $A$ is called \emph{hyperhypersimple} if for all computable functions $f$ such that $\{W_{f(n)}:n\in\N\}$ is a pairwise disjoint family of finite sets, there exists $n\in\N$ such that $W_{f(n)}\subset A$. We refer the reader to \cite[Chapter~5]{soare2} for more about these properties, including examples.

\begin{theorem}
  \label{thm:nonhhs-is-high}
  Let $A\subset\N$ be a set which is simple and not hyperhypersimple. Then $E_A$ is a dark ceer and $E_A$ is high for the computable FS-jump.
\end{theorem}

\begin{proof}
  It follows from \cite[Proposition~4.5]{gao-gerdes} together with the assumption that $A$ is simple that $E_A$ is dark.
  
  To see that $E_A^{\jump}$ is computably bireducible with $=^{ce}$, first it follows from Proposition~\ref{prop:upperbound} that $E_A^{\jump}\leq\mathord{=}^{ce}$. For the reduction in the reverse direction, since $A$ is not hyperhypersimple, there exists a computable function $f$ such that $\{W_{f(n)}:n\in\N\}$ is a pairwise disjoint family of finite sets and for all $n\in\N$ we have $W_{f(n)}\cap A^c\neq\emptyset$. Now given an index $e$ we compute an index $g(e)$ such that $\phi_{g(e)}$ is an enumeration of the set $\bigcup\{W_{f(n)}:n\in W_e\}$. Then since the $W_{f(n)}$ are pairwise disjoint and meet $A^c$, we have $W_e=W_{e'}$ if and only if $A^c\cap\ran\phi_{g(e)}$ and $A^c\cap\ran\phi_{g(e')}$ are distinct subsets of $A^c$. It follows that $e\mathrel{=^{ce}}e'$ if and only if $g(e)\mathrel{E_A^{\jump}}g(e')$, as desired.
\end{proof}

On the other hand, there also exist dark ceers $E$ such that $E$ is not high for the computable FS-jump. In order to state the results, we recall from \cite[Chapter~X]{soare} that a c.e.\ subset $A\subset\N$ is said to be \emph{maximal} if $A^c$ is infinite and for all c.e.\ sets $W$ either $W\setminus A$ or $W^c\setminus A$ is finite. We further note that if $A$ is maximal then it is hyperhypersimple.

\begin{theorem}
  \label{thm:maximal}
  Let $A$ be a maximal set. If $B$ is a c.e. set with $B\subsetneq A$, then $E_A^{\jump} < E_B^{\jump}$. In particular, $E_A$ is not high for the computable FS-jump.
\end{theorem}

The proof begins with several preliminary results, which may be of independent value.

\begin{lemma}
  \label{lem:reverse}
  If $A,B$ are c.e.\ sets and $B\subset A$, then $E_A^{\jump} \leq E_B^{\jump}$.
\end{lemma}

\begin{proof}
  If $B$ is non-hyperhypersimple, then the result follows immediately from Proposition~\ref{prop:upperbound} and Theorem~\ref{thm:nonhhs-is-high}. If $B$ is hyperhypersimple, then by \cite[X.2.12]{soare} there exists a computable set $C$ such that $B\cup C=A$. Let $b\in B$ be arbitrary, and define
  \[f(n)=\begin{cases}b&n\in C\\n&n\notin C\end{cases}
  \]
  It is easy to see that $f$ is a computable reduction from $E_A$ to $E_B$, and hence by Proposition~\ref{prop:monotone}(c) we have $E_A^{\jump}\leq E_B^{\jump}$ as desired.
\end{proof}

In the next lemma we will use the following terminology about a function $f\colon\N\to\N$. We say that $f$ is \emph{$=^{ce}$-invariant} if $W_e=W_{e'}$ implies $W_{f(e)}=W_{f(e')}$, that $f$ is \emph{monotone} if $W_{e'}\subset W_{e}$ implies $W_{f(e')}\subset W_{f(e)}$, and that $f$ is \emph{inner-regular} if
\begin{equation}
  \label{eq:inner-regular}
  W_{f(e)}=\bigcup\left\{W_{f(e')} : W_{e'}\subset W_e\text{ and }W_{e'}\text{ is finite}\right\}\text{.}
\end{equation}
We are now ready to state the lemma. 

\begin{lemma}
  \label{lem:monotone}
  If $f$ is a computable function, the properties $=^{ce}$-invariant, monotone, and inner-regular are all equivalent.
\end{lemma}

\begin{proof}
  It is clear that inner-regular implies monotone, and monotone implies $=^{ce}$-invariant. We therefore need only show that $=^{ce}$-invariant implies inner-regular. Assume that $f$ is $=^{ce}$-invariant. Then \cite[Lemma~4.5]{coskey-hamkins-miller} gives that $f$ is monotone, so we have $W_{f(e)} \supset \bigcup\left\{W_{f(e')} : W_{e'}\subset W_e\text{ and }W_{e'}\text{ is finite}\right\}$.
  
  For the subset inclusion of Equation~\eqref{eq:inner-regular}, we assume that $x\in W_{f(e)}$ and aim to show that there exists $e'$ such that $W_{e'}\subset W_e$, $W_{e'}$ is finite, and $x\in W_{f(e')}$. 
  For any $e$, let $W_{e,s}=\{n : n< s \wedge \varphi_{e,s}(n) \downarrow\}$ be the partial enumeration of $W_e$ at stage $s$, so each $W_{e,s}$ is finite.
  We can use the Recursion Theorem to find an index $e'$ which satisfies the following:
  \[ W_{e',s} = \begin{cases}
  W_{e,s} & \text{if $x \notin W_{f(e'),s}$} \\
  W_{e,s'} & \text{if $s' \leq s$ is least with $x \in W_{f(e'),s'}$.}
  \end{cases}
  \]
  
  We must show that $W_{e'}\subset W_e$, $W_{e'}$ is finite, and $x\in W_{f(e')}$. It is clear that $W_{e'}\subset W_e$. To show that $x\in W_{f(e')}$, assume to the contrary that $x\notin W_{f(e')}$. Then $W_{e',s}=W_{e,s}$ for all $s$, that is, we would have $W_{e'}=W_e$. Since $f$ is $=^{ce}$-invariant, we would have $W_{f(e')}=W_{f(e)}$. Our assumption that $x\in W_{f(e)}$ would therefore imply that $x\in W_{f(e')}$ after all.
  
  Now that we know $x\in W_{f(e')}$, we know that there is $s$ with $x \in W_{f(e'),s}$. This means that for all $s$ we have $W_{e',s}=W_{e,s'}$ for the least such $s'$, so $W_{e'}=W_{e,s'}$ is finite, as desired.
\end{proof}

We note that the same conclusions remain true if we replace $=^{ce}$-invariance by $\id^{\jump}$-invariance, i.e., $f$ preserves equality of ranges rather than of domains. Thus we can apply this result to reductions among computable jumps.

\begin{corollary}
  Let $A$ be a maximal set. If $E^{\jump} \leq E_A^{\jump}$, then any $E$-invariant c.e.\ set contains either finitely or cofinitely many $E$-classes. In particular, if $E_B^{\jump} \leq E_A^{\jump}$ then $B$ is maximal.
\end{corollary}

\begin{proof}
  Let $f$ be a computable reduction from $E^{\jump}$ to $E_A^{\jump}$. We can assume without loss of generality that for all $e$, $\ran\phi_{f(e)}$ is $E_A$-invariant. Indeed, we may modify $f$ to ensure that if $\phi_{f(e)}$ enumerates any element of $A$ then $\phi_{f(e)}$ enumerates the rest of $A$ too. Hence we can assume that $f$ is $=^{ce}$-invariant. If $W=\ran \phi_e$ is an $E$-invariant c.e.\ set, then $R=\ran \phi_{f(e)}$ is an $E_A$-invariant c.e.\ set, hence either $R \setminus A$ is finite or $R$ is cofinite. If $R$ is cofinite, then $W$ must contain all but finitely many $E$-classes, or else there would be an infinite increasing chain of $E^{\jump}$-inequivalent c.e\  sets containing $W$ which must map to an infinite increasing chain of $E_A^{\jump}$-inequivalent c.e\  sets, which is impossible. Suppose instead that $R \setminus A$ is finite. Then by inner-regularity and $E_A$-invariance, there must be a finite set $F=\ran \phi_{e_0} \subset W$ so that $R= \ran \phi_{f(e_0)}$. But then $e_0 \mathrel{E^{\jump}} e$, so $W$ must contain only finitely many $E$-classes.
\end{proof}

We will use the following lemma, well-known in descriptive set theory as a consequence of the effective Reduction Property for the pointclass $\Sigma^0_1$.
\begin{lemma}
\label{lem:effective-reduction}
Let $A_n$ be a uniformly c.e. sequence of c.e. sets. Then there is a uniformly c.e. sequence of c.e. sets $\tilde{A}_n$ so that $\tilde{A}_n \subset A_n$ for each $n$, $\tilde{A}_n \cap \tilde{A}_m = \emptyset$ for $n \neq m$, and $\bigcup_n \tilde{A}_n = \bigcup_n A_n$.
\end{lemma}

\begin{proof}
  Let $f$ be a computable function with $A_n=W_{f(n)}$ for each $n$. Let $\tilde{A}_n.           = \{i : \exists s (\varphi_{f(n),s}(i)\downarrow \wedge (\forall m < n) (\forall t \leq s) \varphi_{f(m),t}(i) \uparrow )\}$.
  \end{proof}
 
We now give the main ingredient to the proof of Theorem~\ref{thm:maximal}.

\begin{definition}
  An equivalence relation $E$ is \emph{self-full} if whenever $f$ is a computable reduction from $E$ to $E$, then the range of $f$ meets every $E$ class.
\end{definition}

Letting $\id_n$ denote the identity equivalence relation on $\{0,\ldots,n-1\}$, $E$ being self-full is equivalent to $E \oplus \id_1 \not\leq E$, so is preserved under computable bireducibility. In the following, we say $f$ and $h$ are \emph{$E$-equivalent} if $f(n)\mathrel{E}h(n)$ for all $n$. We say that $h$ is \emph{induced by a finite support permutation of the $E_A$-classes} when there is an $E_A$-invariant permutation $\pi$ with finite support so that $\ran\phi_{h(e)} = \{\pi(n) : n \in \ran\phi_e\}$ for all $e$ so that $\ran\phi_e$ is $E_A$-invariant.

\begin{lemma}
  \label{lem:full}
  If $A$ is maximal then $E_A^{\jump}$ is self-full. In fact, if $f$ is a computable reduction from $E_A^{\jump}$ to itself, then $f$ is $E_A^{\jump}$-equivalent to a function $h$ induced by a finite support permutation of the $E_A$-classes.
\end{lemma}

\begin{proof}
  Suppose $f$ is a computable reduction from $E_A^{\jump}$ to itself. We can assume without loss of generality that for all $e$, $\ran\phi_{f(e)}$ is $E_A$-invariant. Indeed, we may modify $f$ to ensure that if $\phi_{f(e)}$ enumerates any element of $A$ then $\phi_{f(e)}$ enumerates the rest of $A$ too. Having done so, we introduce the following mild abuse of notation: if $R=\ran\phi_e$ then we will write $f(R)$ for $\ran\phi_{f(e)}$. Due to our assumption about $f$, this notation is well-defined. Note that we thus have that $f$ is $=^{ce}$-invariant as well, and thus monotone and inner-regular.
  
  We will exploit the following consequence of the monotonicity of $f$ several times: If $C$ and $D$ are c.e. sets with $C \subset D$ and $D \setminus C$ finite and disjoint from $A$, then $|f(D)\setminus f(C)| \geq |D \setminus C|$. This follows since there is a chain of length $|D \setminus C|+1$ of $E_A^{\jump}$-inequivalent sets between $C$ and $D$, which must map to a chain of $E_A^{\jump}$-inequivalent sets between $f(C)$ and $f(D)$. Similarly, if $C$ is cofinite with $A \subseteq C$, then $|f(C)^c| \geq |C^c|$. The maximality of $A$ then also implies that if $C \setminus A$ is finite then so is $f(C) \setminus A$. 
  
  The heart of the proof will be to show that there is a finite support permutation $\pi$ of $\N$ such that for any c.e.\ set $R$, we have $f(R)=\{\pi(n):n\in R\}$. In particular, this implies that $f$ meets every $E_A^{\jump}$ class, as desired. We begin be seeing that the range of $f$ is almost covered by the images of singletons.
  
  \begin{claim}
    There is a finite set $C$ such that $f(A)\cup(f(\N)\setminus\bigcup_n f(\{n\}))\subset f(C)$.
  \end{claim} 
  \begin{claimproof}
  First, observe that $\bigcup_n f(\{n\})$ is an infinite c.e.\ set (here we tacitly select indices for $\{n\}$ uniformly), and hence intersects $A$, and thus contains $A$. Moreover, there are infinitely many $n$ so that $f(\{n\})$ intersects $A$; otherwise we could omit such $n$ and have an infinite c.e. set disjoint from $A$. Thus $\{n : f(\{n\}) \cap A \neq \emptyset\}$ is an infinite c.e.\ set, and so intersects $A$. Hence there is $n \in A$ with $f(\{n\}) \cap A \neq \emptyset$, so $A\subset f(A)$. Next, since the sets $f(\{n\})$ are distinct for $n \notin A$, we have $\left( \bigcup_n f(\{n\}) \right) \setminus A$ infinite, so the maximality of $A$ implies $\bigcup_n f(\{n\})$ is cofinite. Hence by the inner-regularity of $f$ we can find a finite set $C$ such that $f(A)\cup(f(\N)\setminus\bigcup_n f(\{n\})\subset f(C)$.
  \end{claimproof}
    
  We next see that we will be able to select distinct elements from the images of singletons.    
  \begin{claim} For $n\notin A\cup C$, we have $f(\{n\})\setminus(f(C)\cup\bigcup_{m\neq n}f(\{m\}))\neq\emptyset$.
  \end{claim}
  \begin{claimproof}
  Let $n\notin A\cup C$. Since $f$ is a reduction and is monotone, we can find $x\in f(\N)\setminus f(\N\setminus\{n\})$. Using the definition of $C$, the fact that $n\notin C$, and monotonicity, we have $f(\N)\setminus\bigcup_mf(\{m\})\subset f(C)\subset f(\N\setminus\{n\})$. In particular, $x\notin f(C)$ and therefore $x\in\bigcup_mf(\{m\})$. Again by monotonicity, $x\notin\bigcup_{m\neq n}f(\{m\})$, so we must have $x\in f(\{n\})$, completing the claim.
  \end{claimproof}
  
  We now construct a first approximation to the desired permutation $\pi$. 
  
  \begin{claim} There is a finite support permutation $\sigma$ such that for $n\notin A\cup C$ we have $\sigma(n)\in f(\{n\})\setminus A$. 
  \end{claim}
  \begin{claimproof}
  From the previous claim, Lemma~\ref{lem:effective-reduction} gives a uniformly c.e.\ sequence $B_n$ of pairwise disjoint sets such that $B_n\subset f(\{n\})$ and $\bigcup_n B_n = \bigcup_n f(\{n\})$, so that for $n\notin A\cup C$ we have $B_n\setminus(f(C)\cup\bigcup_{m\neq n}f(\{m\}))\neq\emptyset$; in particular $B_n \setminus f(C) \neq \emptyset$. We may shrink $B_n$ so that $B_n$ is disjoint from the finite set $f(C)\setminus A$ for all $n$; we may further shrink $B_n$ uniformly to a set $\tilde{B}_n$ so that $\tilde{B}_n\setminus A$ is a singleton for all $n\notin A\cup C$. Note that we do not claim or require that this singleton is not in $\bigcup_{m\neq n}f(\{m\})$, but distinct $n$'s not in $A \cup C$ will produce distinct singletons.
  
  Since $C \setminus A$ is finite, we have that $f((C\setminus A)^c)$ is cofinite, and monotonicity of $f$ implies that $|f((C\setminus A)^c)^c|\geq|C\setminus A|$ as discussed above. We may then let $p$ be any injection from $C\setminus A\to f((C\setminus A)^c)^c$. We now define:
  \[G_n=\begin{cases}
    A             & n\in A \\
    A\cup\{p(n)\} & n\in C\setminus A \\
    A\cup \tilde{B}_n     & n\notin A\cup C
  \end{cases}.
  \]
  Observe that $G_n$ is a uniformly c.e.\ sequence, since we may first check if $n \in C \setminus A$; if not, we enumerate $\tilde{B}_n$ into $G_n$ until we see $n$ enumerated in $A$ (if ever), at which point we enumerate $A$ (which will then contain $\tilde{B}_n$) into $G_n$. We define $\sigma$ as follows:
  \[\sigma(n)=\begin{cases}
    n&n\in A\\
    \text{the unique element of $G_n\setminus A$} & n\notin A
  \end{cases}
  \]
  This completes the definition of $\sigma$. Note that $\sigma(n) \in f(\{n\})$ for $n \notin (C \setminus A)$, and $\sigma(n) \notin A$ for $n \notin A$, as required. We do not claim \emph{a priori} that $\sigma$ is computable (although this will follow later), but we can use the sequence $G_n$ to obtain effectiveness. Define the function $g(R)=\bigcup_{n\in R}G_n$, which is computable in the indices. 

  We check that $\sigma$ is a permutation with finite support. It is immediate from the construction that $\sigma$ is injective. To show $\sigma$ is surjective, assume $k$ is not in the range of $\sigma$. Then the sequence $R_0=A\cup\{k\}$ and $R_{n+1}=g(R_n)$. Since $k\notin A$ we have that $R_n\setminus A$ is a singleton for all $n$. Moreover the singletons are distinct since $\sigma$ is injective and none of the singletons can equal $k$ for $n>0$. Applying Lemma~\ref{lem:effective-reduction} to the sequence $R_n$, we obtain a uniformly c.e.\ sequence of nonempty pairwise disjoint sets, all meeting $A^c$. This contradicts that $A$ is hyperhypersimple (see \cite[Exercise~X.2.16]{soare}). To see that $\sigma$ has finite support, first note that $\sigma$ cannot have an infinite orbit. Otherwise, we could similarly produce a uniformly c.e. sequence (using the function $g$) which contradicts that $A$ is hyperhypersimple. If $\sigma$ had infinitely many nontrivial orbits, let $R=\{n:(\exists k\geq n)\;n\in G_k\}$. Then $A\subset R$, and for $n\notin A$ we have $n\in R$ when $n$ is the least element of its orbit and $n\notin R$ when $n$ is the greatest element of a nontrivial orbit. Thus $R\setminus A$ is infinite and co-infinite, again contradicting that $A$ is maximal.
  \end{claimproof}
  
  We are now ready to construct $\pi$ as follows. Let $\tilde C=(C\setminus A)\cup\supp(\sigma)$. If $R$ is disjoint from $C\setminus A$ then $g(R)\subset f(R)$, therefore if $R$ is disjoint from $\tilde C$ we have $R\subset f(R)$. By monotonicity of $f$, if $R$ is disjoint from $\tilde C$ and cofinite, then by the observation above we have $|f(R)^c| \geq |R^c|$, so $R=f(R)$. In particular $f(\tilde{C}^c)=\tilde{C}^c$. For any $k\in\tilde{C}$, since $k \notin A$ we have that $\tilde C^c\cup\{k\}$ is $E_A^{\jump}$-inequivalent to $\tilde{C}^c$; therefore we must have that $f$ sends $\tilde C^c\cup\{k\}$ to $\tilde C^c\cup\{\pi(k)\}$ for some $\pi(k)\in\tilde C$ since the complement of $f(\tilde C^c\cup\{k\})$ must be at least as large as the complement of $\tilde C^c\cup\{k\}$, but must be smaller than $\tilde{C}$. This defines $\pi$ on the finite set $\tilde{C}$; as $\pi$ is injective it is a permutation of $\tilde{C}$. Additionally define $\pi$ to be the identity on $\tilde C^c$. This completes the definition of $\pi$. 
  
  We have that $\pi$ is an $E_A$-invariant permutation with finite support; it remains to verify that $\pi$ induces the desired function. Let $h$ be a computable function such that $h(R)=\{\pi(n):n\in R\}$. 
  
  \begin{claim} For all $E_A$-invariant c.e. sets $R$ we have $f(R)=h(R)$. 
  \end{claim}
  \begin{claimproof}
  We first establish this for sets of the form $\N \setminus \{n\}$ for $n \notin A$.
  Note that if $F \subset \tilde{C}$ then we have $f(\tilde{C}^c \cup F) = \tilde{C}^c \cup \{\pi(n):n \in F\}$ from monotonicity and the observation earlier; in particular we see that $f(\N)=\N$ and $f(\tilde{C}^c)=\tilde{C}^c$. From this, we see that if $R$ is any cofinite set containing $A$, then $|f(R)^c|=|R^c|$. For any given $n \notin A$, we claim that $f(\N \setminus \{n\})=\N \setminus \{\pi(n)\}$. We know this already for $n \in \tilde{C}$. If $n \in \tilde{C}^c$ then $f(\N \setminus \{n\})=\N \setminus \{\pi(k)\}$ for some $k$; this $k$ can not be in $\tilde{C}$ since $f(\N \setminus \{k\})=\N \setminus \{\pi(k)\}$ and $\N \setminus \{n\}$ and $\N \setminus \{k\}$ are $E_A^{\jump}$-inequivalent. But since $\tilde{C}^c\setminus \{n\} \subset f(\N \setminus \{n\})$ we must have $\pi(k)=n=\pi(n)$ as claimed.
  
    Now let $R$ be any infinite $E_A$-invariant c.e.\ set, so $R$ contains $A$. Then $R$ is the intersection of the sets $\N \setminus \{n\}$ for $n \notin R$, so $f(R) \subset h(R)$ for all such $R$ by monotonicity. Suppose there were an infinite $E_A$-invariant c.e.\ set $R$ and some $k \in h(R) \setminus f(R)$. Then $k=\pi(n)$ for some $n \in R \setminus A$, so $k \notin f(A \cup \{n\}) \subset h(A \cup \{n\})=A \cup \{k\}$ and thus $f(A \cup \{n\})=A$, contradicting $A \subset f(A)$.
    
    Finally, we consider finite $R$. If $n \notin A$, then $f(\{n\}) \subset f(A \cup \{n\}) = h(A \cup \{n\}) = A \cup \{\pi(n)\}$. We can not have $f(\{n\}) =A$, so $f(\{n\})=\{\pi(n)\}$. Let $R$ be any finite $E_A$-invariant c.e.\ set, so $R$ is disjoint from $A$. Then $h(R) \subset f(R)$. We also have that $f(R) \subset f(A \cup R)=h(A \cup R)=A \cup h(R)$, so we must have that $f(R) =h(R)$.
    \end{claimproof}
    Hence, $f$ is $E_A^{\jump}$-equivalent to $h$, completing the proof of the lemma.
\end{proof}

\begin{proof}[Proof of Theorem~\ref{thm:maximal}]
  Let $A$ be maximal and $B\subsetneq A$. Fix any $a\in A\setminus B$. We observe that $E_{A\setminus\{a\}}$ is computably bireducible with $E_A\oplus\id_1$. Therefore by Proposition~\ref{prop:monotone}(c) we have that $E_{A\setminus\{a\}}^{\jump}$ is computably bireducible with $E_A^{\jump}\times\id_2$.
  
  Now assume towards a contradiction that $E_B^{\jump}\leq E_A^{\jump}$. Then by Lemma~\ref{lem:reverse} we have $E_{A\setminus\{a\}}^{\jump}\leq E_A^{\jump}$ and hence by the previous paragraph we have $E_A^{\jump}\times\id_2\leq E_A^{\jump}$. But if $f$ is such a reduction, then by Lemma~\ref{lem:full} the restriction of $f$ to either copy of $E_A^{\jump}$ has range meeting every $E_A^{\jump}$ class. But for a reduction $f$ we cannot have this property true of both copies of $E_A^{\jump}$, so we have reached a contradiction.
\end{proof}

We record here several immediate consequences of Theorem~\ref{thm:maximal} and its proof.

\begin{corollary}
  Let $A,B$ be maximal sets.
  \begin{itemize}
    \item If $a\in A$ then $E_A^{\jump} < E_{A\setminus\{a\}}^{\jump}$, and if $b\notin A$ then $E_{A\cup\{b\}}^{\jump} < E_{A}^{\jump}$.
    \item If $|A\triangle B|<\infty$, then $E_A^{\jump} \leq E_B^{\jump}$ iff $|B\setminus A|\leq|A\setminus B|$.
    \item If a c.e. set $C$ is contained in a maximal set, then it is contained in a maximal set $D$ such that $E_D^{\jump} < E_C^{\jump}$.
  \end{itemize}
\end{corollary}

We conclude with a small refinement of the second statement of Theorem~\ref{thm:maximal}. Recall that a c.e.\ set $A$ is said to be \emph{quasi-maximal} if it is the intersection of finitely many maximal sets. We refer the reader to \cite[X.3.10]{soare} for more on this notion. In particular, every quasi-maximal set is simple (see \cite[X.3.10(b)]{soare}).

\begin{theorem}
  \label{thm:quasimaximal}
  If $A\subset\N$ is quasi-maximal then $E_A$ is not high for the computable FS-jump.
\end{theorem}

\begin{proof}
  Suppose towards a contradiction that $\mathord{=}^{ce}\leq E_A^{\jump}$. By Proposition~\ref{prop:upperbound}, $E_A^{\jump}$ is reducible to the restriction of $=^{ce}$ to the $E_A$-invariant sets, so by composing reductions there exists a computable reduction $f$ from $=^{ce}$ to $=^{ce}$ such that for all $e$, $W_{f(e)}$ is $E_A$-invariant. Since $A$ is simple, it follows that for all $e$ we have $A\subset W_{f(e)}$ iff $W_{f(e)}$ is infinite and $A\cap W_{f(e)}=\emptyset$ iff $W_{f(e)}$ is finite.
  
  We claim that we may find such an $f$ so that for all $e$ we have $A\subset W_{f(e)}$. We first show that there exists $e_0$ such that $W_{e_0}$ is finite and $A\subset W_{f(e_0)}$. Let $e$ be any index such that $W_e=\N$. By Lemma~\ref{lem:monotone}, $f$ is inner-regular. Since $f$ is a reduction, it follows from Equation~\ref{eq:inner-regular} that $W_{f(e)}$ is infinite and hence $A\subset W_{f(e)}$. Further examining Equation~\ref{eq:inner-regular}, together with the last sentence of the previous paragraph, we conclude there exists $e_0$ as desired.
  Let $g$ be a computable function such that $W_{g(e)}=W_{e_0}\cup\{\max(W_{e_0})+x:x\in W_e\}$. Then replacing $f$ with $f\circ g$ completes the claim.
  
  It follows from the claim, together with the fact that $f$ is monotone, that the lattice of c.e.\ sets modulo finite may be embedded into the lattice of c.e.\ sets containing $A$ modulo finite. But the former lattice is infinite, and by \cite[X.3.10(a)]{soare} the latter lattice is finite, a contradiction.
\end{proof}

\section{Additional remarks and open questions}

We close with some open questions and directions for further investigation.

\begin{question}
  For a c.e.\ set $A$, when is $E_A^{\jump}$ bireducible with $=^{ce}$?
\end{question}

By Theorem \ref{thm:nonhhs-is-high} if $A$ is not hyperhypersimple then $E_A$ is high for the jump, and by Theorem~\ref{thm:quasimaximal} if $A$ is quasi-maximal then $E_A^{\jump}<\mathord{=}^{ce}$. The question is, if $A$ is hyperhypersimple but not quasi-maximal, is $E_A^{\jump}$ high for the computable FS-jump? One construction of such a set is given in an exercise in \cite[IX.2.28f]{odifreddi}.
 
We do not know whether the choice of notation for a countable ordinal affects the iterated jump.
 
\begin{question}
  If $a,b \in \mathcal{O}$ with $|a|=|b|$, is $E^{\jump a}$ computably bireducible with $E^{\jump b}$?
\end{question}
 
Although we saw that every hyperarithmetic set is many-one reducible to some jump of the identity, we do not know if every hyperarithmetic equivalence relation is computably reducible to some iterated jump of the identity.

\begin{question}
  If $E$ is hyperarithmetic, is there $a \in \mathcal{O}$ with $E \leq \id^{\jump a}$?
\end{question}

For $E$ hyperarithmetic, we have $e \mathrel{E} e'$ iff $[e]_E=[e']_E$, so that $E$ is computably reducible to the relativized version of $=^{ce}$, denoted $=^{ce,E}$, considered in \cite{bard}. This question is then equivalent to asking if these relativized equivalence relations with hyperarithmetic oracles are computably reducible to iterated jumps of the unrelativized $=^{ce}$.

We also note that, unlike the case of the classical Friedman--Stanley jump, the equivalence relation $E_1$ is not an obstruction. Recall that $E_1$ may be defined on $\mathcal{P}(\N)$ by setting $A \mathrel{E_1} B$ when $A^{[n]}=B^{[n]}$ for all but finitely many $n$, and that $E_1$ is not Borel reducible to any iterated Friedman--Stanley jump of equality. 

\begin{proposition}
  $E_1^{ce} \leq \left(=^{ce}\right)^{\jump}$.
\end{proposition}

\begin{proof}
  Given $e$, let $g(e)$ be such that $\phi_{g(e)}(\langle f,m\rangle)$ is an index for an enumeration of the set $\bigcup_{n<m} (W_f)^{[n]} \cup \bigcup_{n \geq m} (W_e)^{[n]}$. Then $e \mathrel{E_1^{ce}} e'$ if and only if $g(e) \mathrel{\left(=^{ce}\right)^{\jump}} g(e')$.
\end{proof}

We can also ask what other fixed points exist besides $\cong_{\mathcal{T}}$. We note that there is no known characterization of fixed points of the classical Friedman--Stanley jump.

\begin{question}
  Characterize the fixed points of the computable FS-jump.
\end{question}  

We used the relations $\subseteq_E$ in establishing properness of the computable FS-jump jump, but the operation $E \mapsto \subseteq_E$ can be applied to other relations and it may be of interest to study its effect on partial orders.

\begin{question}
What can be said about the mapping $E \mapsto \subseteq_E$ as an operation on computable partial orders?
\end{question}

Finally, we can ask about when the computable FS-jump of an equivalence relation fails to be above the identity relation. The proof of Theorem~\ref{thm:verydarkarithmetic} shows that there is a $\Delta^0_4$ equivalence relation $E$ with infinitely many classes so that $\id\not\leq E^{\jump}$, and Theorem~\ref{thm:lowerbound} shows that there is no $\Sigma^0_1$ such $E$, but we do not know if there can be such an $E$ which is, e.g., $\Sigma^0_2$ or $\Sigma^0_3$.

\begin{question}
What is the least complexity of an equivalence relation $E$ with infinitely many classes such that $\id\not\leq E^{\jump}$?
\end{question}

\bibliographystyle{plain}

\end{document}